\documentclass[11pt]{amsart}
\usepackage{amstext,amsmath,amssymb,amsfonts}
\usepackage{amsthm}
\usepackage{xcolor}
\usepackage{soul}
\usepackage[utf8]{inputenc}
\usepackage[english]{babel}
\usepackage[T1]{fontenc} 
\usepackage{braket,mathtools}
\usepackage{hyperref,csquotes}
\usepackage[english,capitalize]{cleveref}
\usepackage[a4paper,twoside,top=0.74in, bottom=0.74in, left=0.68in, right=0.68in]{geometry}
\usepackage{tikz}
\usepackage{enumitem}
\usetikzlibrary{patterns}

\definecolor{newblue}{rgb}{0.2, 0.3, 0.85}
\hypersetup{colorlinks=true, linkcolor=newblue, citecolor=newblue, urlcolor  = newblue}

\usepackage[maxbibnames=99,backend=biber, sorting=nyt, doi=false, url=false, isbn=false, style=alphabetic]{biblatex}
\numberwithin{equation}{section}

\newcommand{\R}{\mathbb{R}}
\newcommand{\RR}{\mathbb{R}}

\renewcommand{\epsilon}{\varepsilon}
\newcommand{\eps}{\varepsilon}

\newcommand{\Ric}{{\rm Ric}}

\newcommand{\Lip}{{\rm Lip}}

\DeclarePairedDelimiter\scal{\langle}{\rangle} 

\newcommand{\D}{\nabla}
\newcommand{\p}{\partial}
\renewcommand{\P}{\mathcal{P}}
\renewcommand{\H}{\mathcal{H}}

\DeclareMathOperator{\loc}{loc}
\renewcommand{\bar}{\overline}
\renewcommand{\leq}{\leqslant}
\renewcommand{\geq}{\geqslant}
\renewcommand{\le}{\leqslant}
\renewcommand{\ge}{\geqslant}
\renewcommand{\tilde}{\widetilde}

\DeclareMathOperator{\biRic}{biRic}
\DeclareMathOperator{\Sect}{Sect}

\theoremstyle{plain}
\newtheorem{theorem}{Theorem}[section]

\newtheorem{lemma}[theorem]{Lemma}
\newtheorem{corollary}[theorem]{Corollary}
\theoremstyle{remark}

\newtheorem{remark}[theorem]{Remark}
\theoremstyle{definition}

\author{Gioacchino Antonelli}
\address{Gioacchino Antonelli
\hfill\break Courant Institute Of Mathematical Sciences (NYU), 251 Mercer Street, 10012, New York, USA}
\email{ga2434@nyu.edu}

\author{Marco Pozzetta}
\address{Marco Pozzetta
\hfill\break Dipartimento di Matematica, Politecnico di Milano, Via Bonardi 9, 20133, Milano, Italy} \email{marco.pozzetta@polimi.it}

\author{Kai Xu}
\address{Kai Xu
\hfill\break Department of Mathematics, Duke University, Durham, NC 27708, USA}
\email{kai.xu631@duke.edu}

\begin{document}

\title[Spectral splitting theorem]{A sharp spectral splitting theorem}

\date{\today}

\begin{abstract}
    We prove a sharp spectral generalization of the Cheeger--Gromoll splitting theorem. We show that if a complete non-compact Riemannian manifold $M$ of dimension $n\geq 2$ has at least two ends and
    \[
    \lambda_1(-\gamma\Delta+\mathrm{Ric})\geq 0,
    \]
    for some $\gamma<\frac{4}{n-1}$, then $M$ splits isometrically as $\mathbb R\times N$ for some compact manifold $N$ with nonnegative Ricci curvature. We show that the constant $\frac{4}{n-1}$ is sharp, and the multiple-end assumption is necessary for any $\gamma>0$. 
\end{abstract}

\maketitle



\vspace{-15pt}

\section{Introduction}

The celebrated Cheeger--Gromoll splitting theorem \cite{CheegerGromoll} asserts that a complete Riemannian manifold $(M,g)$ having nonnegative Ricci curvature and containing a geodesic line must split isometrically; namely, one has $(M,g)\cong(\R\times N,\mathrm{d}t^2+g_N)$ for some manifold $(N,g_N)$. As a consequence, a complete manifold with nonnegative Ricci curvature and at least two ends must split. In this paper, we give a sharp generalization of the latter statement to manifolds with nonnegative Ricci curvature in the spectral sense.

On a smooth complete $n$-dimensional Riemannian manifold $(M^n,g)$ with $n\geq 2$, we define the function $\mathrm{Ric}:M\to \mathbb R$ by
\[
\mathrm{Ric}(x):=\inf_{\substack{v\in T_xM\\ g(v,v)=1}} \mathrm{Ric}_x(v,v).
\]
Note that $\Ric\in\Lip_{\text{loc}}(M)$. For a constant $\gamma\geq0$, we say that \textit{$M$ satisfies $\lambda_1(-\gamma\Delta+\Ric)\geq0$}
if any of the two following equivalent conditions holds:
\begin{enumerate}[label={(\roman*)}, topsep=2pt, itemsep=1pt]
    \item for all $\varphi\in C^1_c(M)$ it holds $\int_M\gamma|\D\varphi|^2+\Ric\cdot\varphi^2\geq0$,
    \item there exist $\alpha\in (0,1)$ and $u\in C^{2,\alpha}(M)$ such that $u>0$ and $-\gamma\Delta u+\Ric\cdot u\geq0$ on $M$.
\end{enumerate}
The equivalence of the latter two conditions is readily established on compact manifolds, and follows from \cite{Fischer-Colbrie-Schoen} on noncompact manifolds. In fact, in the noncompact case, Item (i) implies the existence of $u>0$ with $-\gamma\Delta u+\Ric\cdot u=0$ \cite{Fischer-Colbrie-Schoen}.

The study of Ricci curvature lower bounds in the spectral sense has recently found connections with the stable Bernstein problem \cite{ChodoshLiBernstein2, CLMS, Mazet}, the geometry of biRicci curvature \cite{ChuLeeZhu24biRicci, Xu24Dimension}, and Ricci curvature bounds in the Kato sense \cite{CMT}. In \cite{AX24}, the first and third named authors proved a sharp volume comparison and Bonnet--Myers theorem for manifolds with positive Ricci curvature in the spectral sense; we refer to \cite{AX24} for a more detailed survey of past results. 

The main result of this paper is the following sharp spectral splitting theorem.

\begin{theorem}\label{thm:mainIntro}
    Let $n\geq 2$, $\gamma<\frac{4}{n-1}$, and let $(M^n,g)$ be an $n$-dimensional smooth complete noncompact Riemannian manifold without boundary. Assume that $M$ has at least two ends and satisfies
    \begin{equation}\label{eq:main_condition}
        \lambda_1(-\gamma\Delta+\Ric)\geq0.
    \end{equation}
    Then $\Ric\geq0$ on $M$. In particular, the manifold splits isometrically as $(M,g)\cong(\R\times N,\mathrm{d}t^2+g_N)$ for some compact manifold $N$ with nonnegative Ricci curvature.
\end{theorem}
With regard to the sharpness of our assumptions, we have the following observations.

\begin{enumerate}
    \item The range $\gamma<\frac{4}{n-1}$ for $\gamma$ is sharp for any $n\geq2$. Indeed, for each $\gamma\geq\frac4{n-1}$ there exists a complete manifold $M^n$ with two ends that satisfies \eqref{eq:main_condition}, and that does not split isometrically. See \cref{rmk:gamma} for details.

    \item For $n\geq3$ and $\gamma>0$, in \cref{rmk:Rn} we construct a small compact perturbation of the flat $\R^n$ that satisfies \eqref{eq:main_condition} but does not split isometrically. Hence, without further assumptions, the spectral splitting theorem does not hold for manifolds containing geodesic lines: in fact, a compact perturbation of $\R^n$ contains infinitely many geodesic lines.
\end{enumerate}

We remark that other types of splitting theorems have been proved in \cite[Theorem 2.1]{LiWangJDG} and \cite[Corollary 5.3 \& Theorem 7.2]{LiWangAsens}, where suitable spectral lower bounds on the curvature imply the splitting of a manifold as a warped product.
\smallskip

The proof of \cref{thm:mainIntro} follows by incorporating the $\mu$-bubble technique of Gromov \cite{GromovPositiveCurvature, GromovFourLectures} and a surface-capturing technique inspired by arguments in \cite{GangLiu, CCE, CEM, ChuLeeZhu24Kahler}.
We also refer to \cite[Section 2]{Zhu23JDG} and \cite[Section 3]{ChodoshliSoapBubbles24} for further introductions to $\mu$-bubbles. We remark that the method used in our proof is different from the one used in the original proof of the splitting theorem \cite{CheegerGromoll} (see also \cite{EschenburgHeintze}). In fact, our method gives a different alternative proof of the classical splitting theorem on manifolds with two ends, $\mathrm{Ric}\geq 0$, and $\inf_{x\in M}|B(x,1)|>0$, see the second part of \cref{rem:ZeroGradientOnPerimeterMinimizer}.

A brief description of our main idea is as follows. We are first inspired by the work of J. Zhu \cite{Zhu23JDG}, where $\mu$-bubbles are used to give an alternative proof of the following result: if $3\le n\le 7$  and there is a complete metric on $\mathbb T^{n-1}\times\RR$ with nonnegative scalar curvature, then it must be flat. The idea of \cite{Zhu23JDG} is to find an area-minimizing hypersurface $\Sigma\subset \mathbb T^{n-1}\times\R$ which arises as a limit of $\mu$-bubbles. Indeed, once such a $\Sigma$ is found, a standard foliation argument can be employed to prove the rigidity. There are two key challenges in this strategy: (1) ensuring the existence of a nonempty limiting hypersurface $\Sigma$, which amounts at preventing the approximating 
$\mu$-bubbles from drifting to infinity, and (2) proving the compactness of $\Sigma$; indeed, each approximating $\mu$-bubble is compact, but the limit hypersurface may not.

Now let us turn to the setting of Theorem \ref{thm:mainIntro}. By \cite{Fischer-Colbrie-Schoen}, we have a function $u$ with
\begin{equation}\label{eq:intro-supersol}
    u>0,\qquad -\gamma\Delta u+\Ric\cdot u = 0.
\end{equation}
Given that $M$ has at least two ends, an analog of Zhu's $\mu$-bubble approximation argument produces a hypersurface $\Sigma\subset M$ that locally minimizes the weighted area $\Sigma\mapsto\int_\Sigma u^\gamma$. However, in our setting, $\Sigma$ might be noncompact, so the foliation argument is no longer available. 

To resolve this issue, we give up the foliation argument and instead argue as follows. We find an approximation scheme that for each $x\in M$ gives a weighted minimal hypersurface $\Sigma$ passing through $x$. Then the stability inequality yields $|\nabla u|=0$ on $\Sigma$, hence arbitrariness of $x$ will imply that $u$ is constant. Thus $\Ric\geq0$, and the splitting follows. To find such a $\Sigma$, we adapt the surface-capturing technique of \cite{CCE, CEM}. The core idea is explained by the following instance: if a metric on $\mathbb T^3$
has positive scalar curvature everywhere except in a set $U$, then any homologically area-minimizing surface must intersect $U$. With this in mind, we perturb the function $u$ in \eqref{eq:intro-supersol} so that $-\gamma\Delta u+\Ric\cdot u>0$ everywhere outside a small ball $B(x,r)$, see \cref{lemma:perturb_new}. Then, due to the stability inequality, a weighted area minimizer must intersect $B(x,r)$. Taking $r\to0$, and letting the perturbation tend to zero, we obtain the desired hypersurface as a limit.
\smallskip

Our \cref{thm:mainIntro} can be used to infer information on the ends of stable minimal hypersurfaces in curved manifolds. In \cite{CaoShenZhu97}, the authors prove that stable minimal hypersurfaces in $\mathbb R^{n+1}$ have exactly one end. In \cite{LiWangCrelle} the authors prove that properly immersed stable minimal hypersurfaces in manifolds with nonnegative sectional curvature either have one end, or they split isometrically and are totally geodesic.

Let us recall that, for orthonormal vectors $v,v'\in T_pM$, the biRicci curvature operator is defined as
\[
\mathrm{biRic}_p(v,v'):=\mathrm{Ric}_p(v,v)+\mathrm{Ric}_p(v',v')-\mathrm{Sect}_p(v\wedge v').
\]
Define the function $\mathrm{biRic}:M\to \mathbb R$ as
\[
\mathrm{biRic}(x):=\inf_{\substack{v,v'\in T_xM \\ \text{$v,v'$ orthonormal}}} \mathrm{biRic}_x(v,v').
\]
Notice that $\mathrm{Sect}\geq 0$ implies $\mathrm{biRic}\geq 0$. The following immediate corollary of \cref{thm:mainIntro} generalizes \cite[Theorem 0.1]{LiWangCrelle} to manifolds with $\mathrm{biRic}\geq 0$ in ambient dimension $n\leq 5$. Contrarily to \cite[Theorem 0.1]{LiWangCrelle}, we do not require the immersion to be proper.

\begin{corollary}\label{cor:RigidityTwoEndedBiRicci}
    Let $(M^n,g)$ be a smooth complete noncompact Riemannian manifold with $\mathrm{biRic}\geq 0$, and assume $n\leq 5$. Let $\Sigma^{n-1}\hookrightarrow M^n$ be a noncompact, complete, connected, immersed, two-sided, stable minimal hypersurface. Then either
    \begin{enumerate}
        \item $\Sigma$ has one end, or
        \item $\Sigma$ splits isometrically as $\mathbb R\times \Sigma'$, where $\Sigma'$ is compact and $\mathrm{Ric}_{\Sigma'}\geq 0$. Moreover, in this case, $\Sigma$ is totally geodesic in $M$, and $\mathrm{Ric}_M(\nu_\Sigma,\nu_\Sigma)=0$, where $\nu_\Sigma$ is a unit normal of $\Sigma$.
    \end{enumerate} 
\end{corollary}

\textbf{Addendum}.\! When we were completing this paper, we learned that the authors of \cite{CatinoMariMastroliaRoncoroni} had independently obtained results that partially overlap with ours, with different techniques. We agreed with them to post on arXiv the results independently at the same time.
\smallskip

\textbf{Organization}.\! After collecting some preliminary results in \cref{sec:Preliminaries}, in \cref{sec:SpectralSplitting} we prove our main theorems. In \cref{sec:Sharpness} we show the sharpness of the assumptions of \cref{thm:mainIntro}.
\smallskip

\textbf{Acknowledgments}. G.A. acknowledges the financial support of the Courant Institute, and the AMS-Simons Travel grant. M.P. is a member of INdAM - GNAMPA. G.A. and M.P. acknowledge the support of the PRIN
 Project 2022E9CF89 - PNRR Italia Domani, funded by EU Program NextGenerationEU. 
Furthermore, G.A. thanks the hospitality of Università di Napoli Federico II, where part of this work was done during his visit in July 2024. The authors thank Otis Chodosh, whose comments in an email exchange motivated us to understand the spectral version of the splitting theorem. 

\section{Preliminaries}
\label{sec:Preliminaries}

The following lemma is a slight modification of \cite[Lemma 2.3]{Zhu23JDG}.

\begin{lemma}\label{lemma:mu_bubble_function}
    For any $\eps\in(0,\frac12)$ there is a smooth function $\bar h_\eps:(-\frac1\eps,\frac1\eps)\to\RR$ such that 
    \begin{enumerate}
        \item on $(-\frac1\eps,-1]\cup[1,\frac1\eps)$ we have
            \begin{equation}
                \bar h'_\eps+\bar h_\eps^2\geq\eps^2,
            \end{equation}
        
        \item on $[-1,1]$ we have
            \begin{equation}
                |\bar h'_\eps+\bar h_\eps^2|\leq C\eps,
            \end{equation}
            for a universal constant $C>0$ ($C=100$ is enough),
        
        \item there holds $\bar h'_\eps<0$ on $(-\frac1\eps,\frac1\eps)$, $\bar h_\eps(0)=0$, $\lim_{x\to\pm\frac1\eps} \bar h_\eps(x)=\mp\infty$,
        
        \item $\bar h_\eps\to0$ smoothly as $\eps\to0$ on any compact subset of $\RR$.
    \end{enumerate}
\end{lemma}
\begin{proof}
    We can take
    \[\bar h_\eps(x):=\eps\eta(x)\coth(1+\eps x)-\eps\big(1-\eta(x)\big)\coth(1-\eps x),\]
    where $\eta$ is a smooth cutoff function with $\eta|_{(-\infty,-1]}\equiv1$, $\eta|_{[1,+\infty)}\equiv0$, $-1\leq\eta'\leq0$, and $\eta(0)=1/2$. 
\end{proof}

The next lemma provides the function we will use to perturb the function $u_0$ satisfying \eqref{eq:intro-supersol}.

\begin{lemma}\label{lemma:perturb_new}
    Let $(M^n,g)$ be a smooth complete noncompact Riemannian $n$-manifold. Let $\gamma>0$, $f\in\Lip_{\loc}(M)$, and assume there exist $\alpha\in (0,1)$, and $u_0\in C^{2,\alpha}(M)$, such that $-\gamma\Delta u_0+fu_0=0$, and $u_0>0$ on $M$. Then for any $x\in M$ and $r\in(0,1)$, there exists $w\in C^{2,\alpha}(M)$ such that
    \begin{enumerate}
        \item $-2u_0\leq w<0$ on $M$,

        \item $-\gamma\Delta w+ f w>0$ on $M\setminus B(x,r)$.
    \end{enumerate}
\end{lemma}
\begin{proof}
    Let $U\Subset B(x,r)$ be a smooth open set, and let $\eta\in\Lip_{\loc}(M)$ be a strictly positive function such that $f+\eta\in C^\infty(M)$. By \cite[Theorem 1]{Fischer-Colbrie-Schoen} we know that $\lambda_1(-\gamma\Delta+f)\geq0$ on $M$, and therefore $\lambda_1(-\gamma\Delta+f+\eta)\geq0$ on $M\setminus U$ as well. Let $\Omega_1\Subset\Omega_2\Subset\ldots\Subset M$ be an exhaustion in smooth precompact open sets. By \cite[Theorem 1]{Fischer-Colbrie-Schoen} again we have $\lambda_1(-\gamma\Delta+f+\eta)>0$ on $\Omega_i\setminus U$ for every $i$. Hence there exists a solution $w_i\in C^{2,\alpha}(\bar\Omega\setminus U)$ to the Dirichlet problem
    \[\left\{\begin{aligned}
        & -\gamma\Delta w_i+fw_i=-\eta w_i\qquad\text{in }\Omega_i\setminus \overline{U}, \\
        & w_i|_{\p\Omega_i}=0,\\
        & w_i|_{\p U}=u_0.
    \end{aligned}\right.\]
    We observe that $w_i>0$ in $\Omega_i \setminus U$. Indeed, testing the equation against $w_i^- :=\min\{w_i,0\} \in \Lip_0(\Omega_i \setminus U)$ yields $0=\int\gamma|\D w_i^-|^2+f(w_i^-)^2+\eta(w_i^-)^2$; since $\lambda_1(-\gamma\Delta+f+\eta)>0$ on $\Omega_i\setminus U$, then $w_i^-=0$ and thus $w_i\ge 0$. The strict inequality $w_i>0$ in $\Omega_i \setminus U$ then follows by the strong maximum principle.
    By an analogous argument, one finds $w_i\leq w_{i+1}$ in $\Omega_i\setminus U$ for any $i$. Moreover $-\gamma\Delta(u_0-w_i)+f(u_0-w_i)=\eta w_i>0$ in $\Omega_i\setminus U$ for any $i$. Multiplying by $(u_0-w_i)^- :=\min\{u_0 -w_i, 0\}$ and arguing as before implies $w_i<u_0$ in $\Omega_i\setminus U$ for any $i$.

    By standard elliptic estimates, using that $(f+\eta)$ is smooth, up to extracting a subsequence, $w_i$ converges in $C^2_{\rm loc}(M\setminus U)$ to a limit function $w'$, which in turn belongs to $C^{2,\alpha}_{\loc}(M\setminus U)$ by elliptic regularity. We thus define $w:=-w'$. From the above construction, we have $-u_0\leq w<0$ and $-\gamma\Delta w+fw=-\eta w>0$ in $M\setminus \overline{U}\Supset M\setminus B(x,r)$. The lemma then follows by extending $w$ inside $U$ keeping $-2u_0\leq w<0$.
\end{proof}

The next lemma proves auxiliary density estimates for $\mu$-bubbles.
We refer the reader to \cite{AmbrosioFuscoPallara, MaggiBook} for the classical theory of sets of finite perimeter in the Euclidean space, which naturally extends to the case of Riemannian manifolds. 

A \textit{set of locally finite perimeter} in a complete $n$-dimensional Riemannian manifold $(M,g)$ is a measurable set $E\subset M$ such that the characteristic function $\chi_E$ has locally bounded variation.
Given such a set $E$, the total variation of the weak gradient of $\chi_E$ is called the perimeter measure, and usually denoted by $P(E,\cdot)$. By the well-known structure theorem of De Giorgi, the perimeter is concentrated on the so-called reduced boundary $\p^*E$. More precisely, we have $P(E,\cdot)=\H^{n-1}\llcorner\p^*E$. For ease of notation, for any set of locally finite perimeter $E\subset M$ and for any Borel set $A\subset M$, in this paper we will denote $|\partial^* E \cap A|:= P(E, A)$ and $|\partial^* E| := P(E,M)$.

\begin{lemma}\label{lem:DensityEstimatesConvergence}
Let $(M^n,g)$ be a smooth noncompact complete $n$-dimensional Riemannian manifold, with $n \ge 2$. Let $\gamma\geqslant 0$, $\eps\in(0,\frac12)$, and let $u:M\to (0,\infty)$ be a function of class $C^{2,\alpha}$, for some $\alpha \in (0,1)$. Let $\phi:M\to \R$ be a surjective smooth proper $1$-Lipschitz function such that $\Omega_0:= \phi^{-1}\left((-\infty,0)\right)$ has smooth bounded boundary. Let $\bar h_\eps$ be given by \cref{lemma:mu_bubble_function}, and let $h_\eps := c \,\bar h_\eps \circ \phi$ for some $c>0$. Define
\begin{equation*}
    \mathcal P_\eps (E) := \int_{\partial^*E} u^\gamma - \int \left( \chi_{E}- \chi_{\Omega_0} \right) h_\eps u^\gamma,
\end{equation*}
for any \textit{admissible} set of locally finite perimeter $E\subset M$, i.e., such that $E \Delta \Omega_0 \Subset \phi^{-1}\left((-1/\eps, 1/\eps)\right)$ up to negligible sets.

Assume that $\Omega_\eps$ is a minimizer for $\mathcal P_\eps$ with respect to admissible compactly supported variations. Then, up to modify $\Omega_\eps$ on a negligible set, the following holds.

\begin{enumerate}
    \item The reduced boundary $\partial^* \Omega_\eps$ 
    is an (open) hypersurface of class $C^{2,\alpha'}$, 
    for some $\alpha'\in(0,1)$, and $\partial \Omega_\eps\setminus \partial^* \Omega_\eps$ 
    is a closed set with Hausdorff dimension less than or equal to $n-8$.

    \item 
    Let $A\subset M$ be a precompact open set, and let $\eps_0\in(0,1/2)$ such that $\overline{A}\subset \phi^{-1}\left((-1/\eps_0, 1/\eps_0)\right)$. Then there exist $r_0,C_d \in (0,1)$ depending on $n,\gamma, \|\log u\|_{L^\infty(A)}, g|_{\overline{A}}, \eps_0, c$ such that for any $\eps<\eps_0$ there holds 
    \[C_d \leqslant\frac{|\partial \Omega_\eps \cap B(y,\rho)|}{\rho^{n-1}} \leqslant C_d^{-1},
    \]
    for any $y \in A \cap \partial \Omega_\eps$ with $B(y,2r_0)\Subset A$, and any $\rho<r_0$.
\end{enumerate}
\end{lemma}

\begin{proof}
Item (1) follows from the Riemannian analogs of \cite[Theorem 27.5]{MaggiBook}, and \cite[Theorem 28.1]{MaggiBook}. We now prove Item (2). We denote here by $C>0$ a generic constant depending on $n,\gamma, \|\log u\|_{L^\infty(A)}$, $g|_{\overline{A}}, \eps_0,c$ that may change from line to line. Let $y \in A$ and $\rho>0$ such that $B(y,\rho)\Subset A$. Since $h_\eps\to0$ smoothly on $\overline{A}$ as $\eps\to0$, we have $\|h_\eps\|_{L^\infty(A)} \le C$ for any $\eps<\eps_0$. Hence for any set $F$ of locally finite perimeter such that $F\Delta \Omega_\eps \Subset B(y,\rho)$, we have
\[
\begin{split}
\int_{\partial^* \Omega_\eps} u^\gamma
\leqslant \int_{\partial^* F} u^\gamma + \int h_\eps u^\gamma ( \chi_{\Omega_\eps} - \chi_F)
&\le 
\int_{\partial^* F} u^\gamma + 
\|h_\eps\|_{L^\infty(A)} \|u\|_{L^\infty(A)}^\gamma |F \Delta \Omega_\eps| \\
&\leq \int_{\partial^* F} u^\gamma  + C |F\Delta \Omega_\eps|.
\end{split}
\]
Using the lower and upper bounds of $u$ on $\bar A$, we estimate
\[\begin{split}
    \big|\partial^*\Omega_\eps \cap B(y,\rho)\big|
    &\leqslant C \left(\big|\partial^* F \cap B(y,\rho)\big| + \big|F\Delta \Omega_\eps\big|^{\frac1n} \big|F\Delta \Omega_\eps\big|^{\frac{n-1}{n}} \right) \\
    &\leqslant C \Big(\big|\partial^* F\cap  B(y,\rho)\big| + \big|B(y,\rho)\big|^{\frac1n}\cdot C_{\rm iso} \big|\partial^* ( F\Delta \Omega_\eps)\big| \Big)
    \\
    &\leqslant C \Big( \big|\partial^* F\cap B(y,\rho)\big| + C \rho \Big[ \big|\partial^* \Omega_\eps\cap B(y,\rho)\big| + \big|\partial^* F \cap B(y,\rho)\big| \Big] \Big),
\end{split}\]
where we used a local isoperimetric inequality with Euclidean exponents for sets contained in $A$ \cite{HebeyBook}, and an Ahlfors-type bound for balls contained in $A$. It follows that there exists $r_0>0$ such that, if $\rho<r_0$, then
\[
\big|\partial^*\Omega_\eps\cap  B(y,\rho)\big| \leqslant C \big|\partial^*F\cap B(y,\rho)\big|,
\]
that is, $\Omega_\eps$ is a $C$-quasiminimal set in balls $B(y,2r_0) \Subset A$ in the sense of \cite[Definition 3.1]{KinnunenShanmugalingamQuasiminimalSets}. The claimed density estimates then follow from
\cite[Lemma 5.1]{KinnunenShanmugalingamQuasiminimalSets} and Item (1).
\end{proof}

\section{The spectral splitting theorem}\label{sec:SpectralSplitting}

In this section we prove \cref{thm:mainIntro} and \cref{cor:RigidityTwoEndedBiRicci}.

\begin{proof}[Proof of Theorem \ref{thm:mainIntro}]
    If $\gamma\leq 0$ the statement reduces to the classical splitting theorem. Hence, from now on, we assume $\gamma>0$. By \cite[Theorem 1]{Fischer-Colbrie-Schoen}, and since $\mathrm{Ric}\in\mathrm{Lip}_{\mathrm{loc}}(M)$, the main condition \eqref{eq:main_condition} implies the existence of a function $u_0\in C^{2,\alpha}(M)$ with
    \begin{equation}
        u_0>0,\qquad -\gamma\Delta u_0+\Ric\cdot u_0=0.
    \end{equation}
    
    Fix an arbitrary $x\in M$. We claim that $|\nabla u_0|(x)=0$. (In fact, it will be clear from the proof that it is possible to find a perimeter minimizer $\Omega$ with $x \in \p\Omega$ and $|\nabla u_0|\equiv0$ on $\p\Omega$; see also \cref{rem:ZeroGradientOnPerimeterMinimizer}). Once this is proved, since $x$ is arbitrary, it follows that $u_0$ is constant on the entire $M$, and thus $\Ric\ge0$ pointwise, and the theorem is reduced to the classical splitting theorem. The rest of the proof is devoted to our claim.

    Since $M$ has more than one end, we can find a smooth surjective proper function $\phi:M\to\RR$ such that
    \begin{equation}\label{eq:phi}
        |\D\phi|\leq1,\qquad |\phi(x)|<1/2,\qquad \Omega_0:=\phi^{-1}\big((-\infty,0)\big)\text{ has smooth boundary.}
    \end{equation}
    The construction is as follows. First, we find a smooth surjective proper $1$-Lipschitz function $\tilde\phi:M\to\RR$. Indeed, given a separating compact hypersurface bounding a distinguished end, $\tilde\phi$ can be obtained by smoothing the signed distance function from such end as in \cite[Lemma 2.1]{Zhu23JDG}. Let $L$ be a regular value of $\tilde\phi$ such that $|\tilde\phi(x) - L| < 1/2$. Hence, $\phi:= \tilde\phi - L$ has the desired properties. 
    
    Let $0<\delta,r\ll1$ be such that
    \begin{equation}\label{eq:radii_constraints}
        B(x,1/\delta)\Supset\phi^{-1}\big([-2,2]\big),\qquad
        B(x,2r)\Subset\phi^{-1}\big([-1,1]\big).
    \end{equation}
    Let us apply Lemma \ref{lemma:perturb_new} to the ball $B(x,r)$, with data $u_0$ and $f=\Ric$. Denote by $w_r$ the resulting function.
    For a parameter $a>0$, set $u_{r,a}:=u_0+aw_r$. There exists a small $a_0=a_0(\delta,w_r,u_0)$ so that for all $a<a_0$, we have
    \begin{equation}\label{eq:choice_a_1}
        u_0/2\leq u_{r,a}\leq u_0,\qquad ||u_{r,a}-u_0||_{C^2(B(x,1/\delta))}<\delta,
    \end{equation}
    and
    \begin{equation}\label{eq:choice_a_2}
        \left\{\begin{aligned}
            & -\gamma\Delta u_{r,a}+\Ric\cdot u_{r,a}>0\qquad\text{in $M\setminus B(x,r)$,} \\
            & -\gamma\Delta u_{r,a}+\Ric\cdot u_{r,a}\geq-\delta \cdot u_{r,a}
            \qquad\text{in $B(x,r)$.} 
        \end{aligned}\right.
    \end{equation}
    In particular, in light of \eqref{eq:radii_constraints} and \eqref{eq:choice_a_2}, we have
    \begin{equation}\label{eq:choice_a_3}
        -\gamma\Delta u_{r,a}+\Ric\cdot u_{r,a}\geq\mu\cdot u_{r,a}\qquad\text{in $\phi^{-1}\big([-2,2]\big)\setminus B(x,2r)$,}
    \end{equation}
    for some constant $\mu=\mu(a,w_r,u_0)>0$.

    For a constant $c_0=c_0(n,\gamma)>0$ that will be defined below in \eqref{eqn:Defnc0}, and for any $\eps<1/10$, let $h_\eps$ be defined as
    \begin{equation}
        h_\eps(x):=c_0^{-1}\bar h_\eps(\phi(x)),\qquad\forall x\in\phi^{-1}\big((-1/\eps,1/\eps)\big),
    \end{equation}
    where $\bar h_\eps$ is as in Lemma \ref{lemma:mu_bubble_function}. In particular, on $\phi^{-1}\big((-1/\eps,-1]\cup[1,1/\eps)\big)$ we have 
    \begin{equation}\label{eq:dh+h2}
        \begin{split}
                c_0h_\eps^2 - |\D h_\eps| &= 
                c_0^{-1}\Big[\bar h_\eps(\phi)^2-|\bar h'_\eps(\phi)||\nabla\phi|\Big]
                \geq c_0^{-1}\eps^2,
        \end{split}
    \end{equation}
    where we used that $\phi$ is $1$-Lipschitz and Lemma \ref{lemma:mu_bubble_function}(1)(3). Then by Lemma \ref{lemma:mu_bubble_function}(2)(3), inside $\phi^{-1}\big([-1,1]\big)$ we have
    \begin{equation}\label{eq:dh+h2_2}
        c_0h_\eps^2-|\D h_\eps|\geq-Cc_0^{-1}\eps,
    \end{equation}
    where $C$ is the constant in \cref{lemma:mu_bubble_function}(2). For $\delta,r,a$ small enough as indicated above, we perform the following $\mu$-bubble argument. For notational simplicity, we denote $u=u_{r,a}$. For any set of finite perimeter $\Omega$ with $\Omega\Delta \Omega_0 \subset \phi^{-1}\big((-1/\eps,1/\eps)\big)$, define the energy
    \begin{equation}\label{eqn:MuBubbleFunctional2}
        \mathcal{P}_\eps (\Omega):=\int_{\partial^*\Omega}u^\gamma-\int(\chi_\Omega-\chi_{\Omega_0}) h_\eps u^\gamma.
    \end{equation}
    As argued in \cite[Proposition 12]{ChodoshliSoapBubbles24}, recalling that $h_\varepsilon$ diverges as $\phi\to \pm 1/\eps$ by \cref{lemma:mu_bubble_function}(3), there exists a minimizer $\Omega_\eps$ for $\mathcal{P}_\eps$ with $\Omega_\eps\Delta \Omega_0 \Subset \phi^{-1}\big((-1/\eps,1/\eps)\big)$. 

    Since $u\in C^{2,\alpha}$, the regularity result in \cref{lem:DensityEstimatesConvergence}(1) applies to $\Omega_\eps$.
    Let $\nu$ denote the outer unit normal at $\partial^*\Omega_\eps$, and let $\varphi\in C^\infty(M)$. For an arbitrary smooth variation $\{F_t\}_{t\in(-t_0,t_0)}$, with $F_0=\Omega_\eps$, and with variation field equal to $\varphi\nu$ at $t=0$, we compute the first variation\footnote{When $n\geq8$, the boundary $\p\Omega_\eps$ may contain a nonempty codimension 8 singular set $\p\Omega_\eps \setminus \p^*\Omega_\epsilon$, see \cref{lem:DensityEstimatesConvergence}(1). However, arguing as in \cite[Appendix A]{AX24}, multiplying $\varphi$ by a cut off function vanishing on $\p\Omega_\eps \setminus \p^*\Omega_\epsilon$, one can calculate the first and second variation of $\mathcal P_\eps$, and carry out computations leading to \eqref{eq:2ndvar_final}. We omit these standard computations.}
\begin{equation}
    0=\frac{\mathrm{d}}{\mathrm{d}t} \mathcal{P}_\eps (F_t)\Big|_{t=0}=\int_{\partial^*\Omega_\eps}\big(H+\gamma u^{-1}u_\nu-h_\eps\big)u^\gamma\varphi,
\end{equation}
where $u_\nu:= \langle\nabla u, \nu \rangle$. Since $\varphi$ is arbitrary we have $H=h_\eps-\gamma u^{-1}u_\nu$. Then, computing the second variation, we obtain
\begin{equation}\label{eq:2nd_var}
    \begin{aligned}
        0 &\leqslant \frac {\mathrm{d}^2}{\mathrm{d}t^2} 
    \mathcal{P}_\eps (F_t)\Big|_{t=0} \\
        &= \int_{\partial^*\Omega_\eps}\Big[-\Delta_{\partial^*\Omega_\eps}\varphi-|\text{II}|^2\varphi-\Ric(\nu,\nu)\varphi-\gamma u^{-2}u_\nu^2\varphi \\
        &\hspace{72pt} +\gamma u^{-1}\varphi\big(\Delta u-\Delta_{\partial^*\Omega_\eps}u-Hu_\nu\big)-\gamma u^{-1}\langle\nabla_{\partial^*\Omega_\eps}u,\nabla_{\partial^*\Omega_\eps}\varphi\rangle - \scal{\nabla h_\eps, \nu} \varphi\Big]u^\gamma\varphi.
    \end{aligned}
\end{equation}
Since $\partial\Omega_\eps$ is compact, 
we can take $\varphi:= u^{-\gamma/2}$ in \eqref{eq:2nd_var}. As a result, we get
\begin{equation}\label{eqn:ControlToDo2}
    \begin{aligned}
        0 &\leqslant \int_{\partial^*\Omega_\eps}\Big[
            -\Delta_{\partial^*\Omega_\eps}( u^{-\gamma/2}) u^{\gamma/2}
            -\gamma u^{-1}\Delta_{\partial^*\Omega_\eps}u
            -\gamma u^{\gamma/2-1} \langle\nabla_{\partial^*\Omega_\eps}u,\nabla_{\partial^*\Omega_\eps}( u^{-\gamma/2})\rangle\Big] \\
        &\qquad\qquad +\Big[-|\text{II}|^2-\gamma u^{-2}u_\nu^2 
        -\gamma Hu^{-1}u_\nu
        - \scal{\nabla h_\eps, \nu} \Big]
        +\Big[\gamma u^{-1}\Delta u-\Ric(\nu,\nu)\Big]\\
        &=: \int_{\partial^*\Omega_\eps} P+Q+R.
    \end{aligned}
\end{equation}

First, integrating by parts we have
\begin{equation}\label{eq:P}
    \int_{\partial^*\Omega_\eps} P=\int_{\partial^*\Omega_\eps}  \frac{\gamma^2}{4}u^{-2} |\nabla_{\partial^*\Omega_\eps} u|^2
    -\gamma u^{-2}|\nabla_{\partial^*\Omega_\eps} u|^2 
    =\int_{\partial^*\Omega_\eps}\Big(\frac{\gamma^2}{4}-\gamma\Big)u^{-2}|\nabla_{\partial^*\Omega_\eps} u|^2,
\end{equation}
and notice that $\frac{\gamma^2}{4}-\gamma = \gamma\big(\frac{\gamma}{4}-1\big)<0$ for $0< \gamma < \frac4{n-1}$.
We now deal with the ``normal'' term $Q$. Denote for simplicity $Y:=u^{-1}u_\nu$. Using the first variation $H=h_\eps-\gamma Y$ and the trace inequality $|\mathrm{II}|^2\geqslant H^2/(n-1)$, we get
\begin{equation}\label{eqn:Central}
    \begin{aligned}
        Q     &\leq -\frac{H^2}{n-1}-\gamma Y^2-\gamma HY+|\D h_\eps|  \\
            &=-\frac{(h_\eps-\gamma Y)^2}{n-1}-\gamma Y^2-\gamma(h_\eps -\gamma Y)Y+|\D h_\eps|  \\
            &= \Big[-\frac{h_\eps^2}{n-1}+\frac{3-n}{n-1}\gamma h_\eps Y+\Big(\frac{n-2}{n-1}\gamma^2-\gamma\Big)Y^2\Big]+|\D h_\eps|.
    \end{aligned}
\end{equation}
Now, computing the discriminant of the $(h_\eps,Y)$-polynomial above, one gets
\begin{equation}\label{eqn:Discriminant}
    \begin{aligned}
        \Delta &= \gamma^2\Big(\frac{3-n}{n-1}\Big)^2+4\frac{\gamma}{n-1}\Big(\frac{n-2}{n-1}\gamma-1\Big)
        =\gamma\Big(\gamma-\frac{4}{n-1}\Big)
        <0.
    \end{aligned}
\end{equation}
Therefore, by continuity, we conclude that 
\begin{equation}\label{eqn:Defnc0}
Q\leq -c_0h_\eps^2-c_0Y^2+|\D h_\eps|,
\end{equation}
for some sufficiently small constant $c_0=c_0(n,\gamma)>0$, which can be chosen so that $c_0 <\gamma-\gamma^2/4$. Notice this $c_0$ is the constant we are using in \eqref{eq:dh+h2}. On the other hand, for the $R$ term, by \eqref{eq:choice_a_2}, and \eqref{eq:choice_a_3} we have
\begin{equation}\label{eq:R}
    R\leq\delta\chi_{B(x,r)}-\mu\chi_{\phi^{-1}([-2,2])\setminus B(x,2r)}.
\end{equation}

Back to \eqref{eqn:ControlToDo2}, and inserting \eqref{eq:dh+h2}, \eqref{eq:dh+h2_2}, \eqref{eq:P}, \eqref{eqn:Defnc0}, and \eqref{eq:R}, we obtain
\begin{equation}\label{eq:2ndvar_final}
    \begin{aligned}
        0\leq \int_{\partial^*\Omega_\eps} P+Q+R &\leq \int_{\p^*\Omega_\eps}\Big[\!-c_0u^{-2}|\D_{\p^*\Omega_\eps}u|^2-c_0u^{-2}u_\nu^2\Big]+\Big[\!-c_0h_\eps^2+|\D h_\eps|\Big]+R \\
        &\leq \int_{\p^*\Omega_\eps}\Big[\!-c_0u^{-2}|\D_{\p^*\Omega_\eps}u|^2-c_0u^{-2}u_\nu^2\Big] \\
        &\qquad - c_0^{-1}\eps^2\big|\p^*\Omega_\eps\setminus\phi^{-1}([-1,1])\big|
        + Cc_0^{-1}\eps\big|\p^*\Omega_\eps\cap\phi^{-1}([-1,1])\big|\\
        &\qquad +\delta\big|\p^*\Omega_\eps\cap B(x,r)\big|
        - \mu\big|\p^*\Omega_\eps\cap\big(\phi^{-1}([-2,2])\setminus B(x,2r)\big)\big|.
    \end{aligned}
\end{equation}
First we note that $\p^*\Omega_\eps\cap\phi^{-1}\big([-1,1]\big)\ne\emptyset$. Indeed, otherwise, \eqref{eq:2ndvar_final} and $B(x,r)\Subset \phi^{-1}\big([-1,1]\big)$ together would imply
\[
0\leqslant -c_0^{-1}\varepsilon^2|\partial^*\Omega_\eps|<0.
\]
Next, we claim that $\p^*\Omega_\eps\cap B(x,2r)\ne\emptyset$ whenever $\eps<C^{-1}c_0\mu$. Indeed, since $\p^*\Omega_\eps\cap\phi^{-1}\big([-1,1]\big)\ne\emptyset$, we have $|\p^*\Omega_\eps\cap\phi^{-1}\big([-2,2]\big)|>0$. Hence if $\p^*\Omega_\eps\cap B(x,2r)=\emptyset$, \eqref{eq:2ndvar_final} would imply
\[
0\leq Cc_0^{-1}\eps\big|\p^*\Omega_\eps\cap\phi^{-1}([-1,1])\big|-\mu\big|\p^*\Omega_\eps\cap\phi^{-1}([-2,2])\big|<0.
\]

Let now $m_0\in \mathbb N$. For every $m\geqslant m_0$, define $\delta_m:=\frac{1}{m}$, and $r_m:=\frac{1}{m}$. If $m_0$ is chosen large enough, according to what said above, for every $\delta_m,r_m$ we can find $a_m,\mu_m>0$ small enough such that \eqref{eq:radii_constraints}, \eqref{eq:choice_a_1}, \eqref{eq:choice_a_2}, and \eqref{eq:choice_a_3} are met. Denote by $u_m=u_0+a_mw_{r_m}$ the perturbed functions. Choosing $\varepsilon_m<\min\{\frac{1}{m},C^{-1}c_0\mu_m\}$, from what we said above we know that there exists $y_m\in \partial^*\Omega_{\varepsilon_m}\cap B(x,2r_m)$.

In particular, $y_m \to x$ as $m\to\infty$. Moreover, by \eqref{eq:choice_a_1}, $\log u_m$ is locally bounded in $L^\infty$, uniformly with respect to $m$, hence the density estimates provided by \cref{lem:DensityEstimatesConvergence}(2) are uniform locally around $x$ as $m\to \infty$. If by contradiction $|\D u_0|(x)>0$, using \cref{lem:DensityEstimatesConvergence}(2), and recalling \eqref{eq:choice_a_1}, we find $\rho,\eta>0$ such that
\begin{equation}\label{eqn:DensityLowerBound}
    \int_{\partial^*\Omega_{\eps_m}\cap B(x,\rho)} u_m^{-2}|\D u_m|^2
    \geq \int_{\partial^*\Omega_{\eps_m} \cap B(y_m,\rho/2)} u_m^{-2}|\nabla u_m|^2
    \geq \eta,
\end{equation}
for any large $m$.

On the other hand, $|\p^*\Omega_{\eps_m} \cap \phi^{-1}([-1,1])|$ has an upper bound independent of $m$. Indeed, this follows by comparing $\P_{\eps_m}(\Omega_{\eps_m})\leq\P_{\eps_m}(\Omega_0)$, which says
\begin{equation}\label{eqn:Minimality}
    \int_{\p^*\Omega_{\eps_m}}u_m^\gamma\leq\int_{\p\Omega_0}u_m^\gamma+\int\big(\chi_{\Omega_{\eps_m}}-\chi_{\Omega_0}\big)h_{\eps_m}u_m^\gamma\leq\int_{\p\Omega_0}u_m^\gamma,
\end{equation}
by the property of $\overline{h}_{\varepsilon_m}$ in \cref{lemma:mu_bubble_function}(3); and by the fact that $u_m$ are uniformly bounded from above and below in $\phi^{-1}\big([-1,1]\big)$ by \eqref{eq:choice_a_1}. Combining \eqref{eq:2ndvar_final} and \eqref{eqn:DensityLowerBound} we finally find that
\[\begin{aligned}
    0< c_0 \eta &\leq \liminf_{m\to\infty}
    Cc_0^{-1}\eps_m\big|\p^*\Omega_{\eps_m}\cap\phi^{-1}([-1,1])\big| +\delta_m\big|\p^*\Omega_{\eps_m}\cap B(x,r_m)\big| =0,
\end{aligned}\]
providing a contradiction.
\end{proof}

\begin{remark}\label{rem:ZeroGradientOnPerimeterMinimizer}
    In the setting and notation of the proof of \cref{thm:mainIntro}, it is possible to show that there exists a perimeter minimizer $\Omega$ with finite perimeter such that $x \in \partial\Omega$ and $|\nabla u_0|\equiv 0$ on $\p \Omega$. Indeed, one can use \eqref{eqn:Minimality} and \cref{lemma:mu_bubble_function}(4) to prove that $\Omega_{\eps_m}$ converges, up to subsequences, to a local minimizer $\Omega$ of the weighted perimeter $E\mapsto \int_{\partial^*E} u_0^\gamma$. Equation \eqref{eq:2ndvar_final} eventually implies $|\nabla u_0|\equiv 0$ on $\partial \Omega$. A posteriori, by constancy of $u_0$, $\Omega$ is in fact a local perimeter minimizer. Finally, a posteriori of the splitting $M\cong \mathbb R\times N$, and the fact that $N$ is compact, it also follows that $\partial\Omega\cong N$ isometrically.
    \smallskip

    Similarly, we remark that the method used in the proof of \cref{thm:mainIntro} gives a different alternative proof of the Cheeger--Gromoll splitting theorem \cite{CheegerGromoll} in case $(M^n,g)$ has at least two ends, $\mathrm{Ric}\geq 0$, and $\inf_{x\in M}|B(x,1)|>0$. We sketch here the argument. When the previous assumptions are in force, in the proof of \cref{thm:mainIntro} we can considerably simplify the argument by choosing $u_0\equiv 1$ and $\mu=a=\delta=0$, in fact avoiding the complications coming from the surface-capturing technique. Similarly to what said at the beginning of this remark, in this case the corresponding $\mu$-bubbles $\Omega_{\varepsilon}$ will converge, as $\varepsilon\to 0$, to a local perimeter minimizer $\Omega$ with finite perimeter (we do not care whether or not $x\in \partial\Omega$). Uniform density estimates on the boundary of perimeter minimizers\footnote{Here we use $\inf_{x\in M}|B(x,1)|>0$ and $\mathrm{Ric}\geq 0$.} (cf., e.g., \cite[Corollary 4.15]{AntonelliPasqualettoPozzettaSemola1}), and the fact that $\Omega$ has finite perimeter imply that $\partial\Omega$ is compact.
    
    Hence $\partial \Omega$ is a (possibly disconnected) minimal hypersurface, thus the splitting follows from a classical Laplacian comparison argument (see, e.g., \cite{Kasue83}). Alternatively, the splitting also follows from a foliation type argument (see, e.g., \cite[Remark 6]{AX24}).
\end{remark}

\begin{proof}[Proof of \cref{cor:RigidityTwoEndedBiRicci}]
    By \cite[Lemma 1.9]{Xu24Dimension} we deduce $\lambda_1(-\Delta_\Sigma + \mathrm{Ric}_\Sigma)\geq 0$. Since $n\leq 5$, we get $\frac{4}{(n-1)-1}> 1$. If $\Sigma$ has one end, the proof is concluded. If instead $\Sigma$ has at least two ends, by applying \cref{thm:mainIntro} to $\Sigma$, we get that $\Sigma$ splits isometrically as $\mathbb R\times \Sigma'$, with $\mathrm{Ric}_{\Sigma'}\geq 0$ and $\Sigma'$ is compact.

    To complete the proof of (2), for every $a>2$ let $\phi_a:\mathbb R\to\mathbb R$ be the function such that: $\phi|_{[-1,1]}\equiv 1$, $\phi|_{\mathbb R\setminus [-a,a]}\equiv 0$, and $\phi$ is linear on $[-a,-1]\cup [1,a]$. Let $f_a(t,\sigma'):=\phi_a(t)$ be defined on $\mathbb R\times \Sigma'$. Let $\p_t$ be the unit tangent vector of $\Sigma$ in the $\R$ direction. By Gauss equation there holds $0=\Ric_\Sigma(\p_t ,\p_t)=-\text{II}^2(\p_t,\p_t)+\Ric_M(\p_t,\p_t)-\Sect(\p_t \wedge \nu_\Sigma)$. Hence testing the stability condition of $\Sigma$ on the function $f_{a}$, we have
    \[0\leq\int|\D f_a|^2-\Big(|\text{II}|^2+\Ric_M(\nu_\Sigma,\nu_\Sigma)\Big)f_a^2
        = \int|\D f_a|^2-\Big(|\text{II}|^2-\text{II}^2(\p_t,\p_t)+\biRic_M(\p_t,\nu_\Sigma)\Big)f_a^2.\]
    Sending $a\to +\infty$, since $\Sigma'$ is compact we get that $\biRic_M(\p_t,\nu_\Sigma)=0$ and $\text{II}=\text{II}(\p_t,\p_t)\mathrm{d}t^2$. Moreover, since $H=0$, this implies $\text{II}=0$ and hence $\mathrm{Ric}_M(\nu_\Sigma,\nu_\Sigma)=0$, as desired.
\end{proof}

\section{Sharpness and examples}\label{sec:Sharpness}

The following remarks show the sharpness of our assumptions.

\begin{remark}[Sharpness of the range of $\gamma$]\label{rmk:gamma}
    Let $n\geq2$ and $\gamma\geq\frac4{n-1}$. Consider the manifold
    \[
    (M,g)=\big(\R\times\mathbb T^{n-1},\mathrm{d}t^2+e^{2t}g_{\mathbb T^{n-1}}\big).
    \]
    Changing variables $y^{-2}:=e^{2t}$, one checks that $M$ is a quotient of the hyperbolic space, hence $\Ric\equiv-(n-1)$. We may directly compute that the function $u=e^{-\frac{n-1}2t}$ satisfies
    \[\begin{aligned}
        -\gamma\Delta u+\Ric\cdot u &= e^{-\frac{n-1}2t}\Big[\gamma\frac{(n-1)^2}4-(n-1)\Big]\geq0.
    \end{aligned}\]
    Hence $M$ satisfies \eqref{eq:main_condition}, has two ends, but it does not split isometrically.
\end{remark}

\begin{remark}[Small compact perturbation of $\R^n$]\label{rmk:Rn}
    Let $n\geq3$ and $\gamma>0$. Then there exists a smooth function $u$ with $u>0$, $\Delta u\leq0$ and $\Delta u<0$ in the Euclidean ball $B(0,1)$. One example is
    \[
    u=u(r):=\int_r^\infty\Big[2^{-n}\eta(s)s+\big(1-\eta(s)\big)s^{1-n}\Big]\,\mathrm{d}s.
    \]
    where $r=|x|$ and $\eta$ is a cutoff function with $\eta|_{[0,1]}\equiv1$, $\eta|_{[2,\infty)}\equiv0$, and $\eta'\leq0$. Indeed, we may calculate
    \[\begin{aligned}
        \Delta u &= u''+\frac{n-1}ru' \\
        &= \Big[-2^{-n}\eta'r-2^{-n}\eta+(n-1)(1-\eta)r^{-n}+\eta'r^{1-n}\Big]+\frac{n-1}r\Big[-2^{-n}\,r\, \eta -(1-\eta)r^{1-n}\Big] \\
        &= -2^{-n}n\eta+\eta'r(r^{-n}-2^{-n}),
    \end{aligned}\]
    which is nonpositive on $\R^n$ since $\eta'\leq0$ and $r^{-n}-2^{-n}\geq0$ when $r\leq2$. Moreover $\Delta u \le -2^{-n}n<0$ when $r<1$. Denote by $g_0$ the Euclidean metric, and let $g$ be a Riemannian metric with $\{g\ne g_0\}\Subset B(0,1/2)$. By continuity, it follows that $-\gamma\Delta_g u+\Ric_g\cdot u\geq0$, provided $||g-g_0||_{C^2}$ is sufficiently small. This shows that any sufficiently small compact perturbation of flat $\R^n$ satisfies the main condition \eqref{eq:main_condition}. Moreover $(\R^n,g)$ contains lines but, suitably choosing $g$, it does not split isometrically.
\end{remark}

\printbibliography[title={References}]

@misc{CatinoMariMastroliaRoncoroni,
      title={Criticality, splitting theorems under spectral Ricci bounds and the topology of stable minimal hypersurfaces}, 
      author={Catino, G. and Mari, L. and Mastrolia, P. and Roncoroni, A.},
      year={2024},
      note = {Preprint to appear},
}

@misc{CLMS,
      title={Stable minimal hypersurfaces in $\mathbf{R}^5$}, 
      author={Chodosh, O. and Li, C. and Minter, P. and Stryker, D.},
      year={2024},
      eprint={2401.01492},
      archivePrefix={arXiv},
      primaryClass={math.DG},
      url={https://arxiv.org/abs/2401.01492}, 
}

@misc{Mazet,
      title={Stable minimal hypersurfaces in $\mathbb R^6$}, 
      author={Mazet, L.},
      year={2024},
      eprint={2405.14676},
      archivePrefix={arXiv},
      primaryClass={math.DG},
      url={https://arxiv.org/abs/2405.14676}, 
}

@misc{ChuLeeZhu24Kahler,
      title={On K\"ahler manifolds with non-negative mixed curvature}, 
      author={Chu, J. and Lee, M.-C. and Zhu, J.},
      year={2024},
      eprint={2408.14043},
      archivePrefix={arXiv},
      primaryClass={math.DG},
      url={https://arxiv.org/abs/2408.14043}, 
}

@misc{ChuLeeZhu24biRicci,
      title={Homological $n$-systole in $(n+1)$-manifolds and bi-Ricci curvature}, 
      author={Chu, J. and Lee, M.-C. and Zhu, J.},
      year={2024},
      eprint={2410.20785},
      archivePrefix={arXiv},
      primaryClass={math.DG},
      url={https://arxiv.org/abs/2410.20785}, 
}

@misc{Xu24Dimension,
      title={Dimension constraints in some problems involving intermediate curvature}, 
      author={Xu, K.},
      year={2024},
      eprint={2301.02730},
      archivePrefix={arXiv},
      primaryClass={math.DG},
      url={https://arxiv.org/abs/2301.02730}, 
}

@misc{CMT,
      title={Kato meets Bakry-\'Emery}, 
      author={Carron, G. and Mondello, I. and Tewodrose, D.},
      year={2023},
      eprint={2305.07428},
      archivePrefix={arXiv},
      primaryClass={math.DG},
      url={https://arxiv.org/abs/2305.07428}, 
}

@misc{GromovFourLectures,
      title={Four Lectures on Scalar Curvature}, 
      author={Gromov, M.},
      year={2021},
      eprint={1908.10612},
      archivePrefix={arXiv},
      primaryClass={math.DG},
      url={https://arxiv.org/abs/1908.10612}, 
}

@article {GangLiu,
    AUTHOR = {Liu, G.},
     TITLE = {3-manifolds with nonnegative {R}icci curvature},
   JOURNAL = {Invent. Math.},
  FJOURNAL = {Inventiones Mathematicae},
    VOLUME = {193},
      YEAR = {2013},
    NUMBER = {2},
     PAGES = {367--375},
      ISSN = {0020-9910,1432-1297},
   MRCLASS = {53C20 (53A10 53C21)},
  MRNUMBER = {3090181},
MRREVIEWER = {David\ J.\ Wraith},
       DOI = {10.1007/s00222-012-0428-x},
       URL = {https://doi.org/10.1007/s00222-012-0428-x},
}

@article {EschenburgHeintze,
    AUTHOR = {Eschenburg, J. and Heintze, E.},
     TITLE = {An elementary proof of the {C}heeger-{G}romoll splitting
              theorem},
   JOURNAL = {Ann. Global Anal. Geom.},
  FJOURNAL = {Annals of Global Analysis and Geometry},
    VOLUME = {2},
      YEAR = {1984},
    NUMBER = {2},
     PAGES = {141--151},
      ISSN = {0232-704X},
   MRCLASS = {53C20},
  MRNUMBER = {777905},
MRREVIEWER = {Thomas\ E.\ Cecil},
       DOI = {10.1007/BF01876506},
       URL = {https://doi.org/10.1007/BF01876506},
}

@unpublished{AntonelliPasqualettoPozzettaSemola1,
      title={Sharp isoperimetric comparison on non-collapsed spaces with lower Ricci bounds}, 
      author={Antonelli, G. and Pasqualetto, E. and Pozzetta, M. and Semola, D.},
      year={2023},
      eprint={2201.04916},
      archivePrefix={arXiv},
      note={To appear: Ann. Sci. \'{E}cole Norm. Sup.},
}

@incollection {GromovPositiveCurvature,
    AUTHOR = {Gromov, M.},
     TITLE = {Positive curvature, macroscopic dimension, spectral gaps and
              higher signatures},
 BOOKTITLE = {Functional analysis on the eve of the 21st century, {V}ol.\
              {II} ({N}ew {B}runswick, {NJ}, 1993)},
    SERIES = {Progr. Math.},
    VOLUME = {132},
     PAGES = {1--213},
 PUBLISHER = {Birkh\"auser Boston, Boston, MA},
      YEAR = {1996},
      ISBN = {0-8176-3855-5},
   MRCLASS = {53C21 (53C20 57R20)},
  MRNUMBER = {1389019},
MRREVIEWER = {Christopher\ W.\ Stark},
       DOI = {10.1007/s10107-010-0354-x},
       URL = {https://doi.org/10.1007/s10107-010-0354-x},
}

@article {CEM,
    AUTHOR = {Chodosh, O. and Eichmair, M. and Moraru, V.},
     TITLE = {A splitting theorem for scalar curvature},
   JOURNAL = {Comm. Pure Appl. Math.},
  FJOURNAL = {Communications on Pure and Applied Mathematics},
    VOLUME = {72},
      YEAR = {2019},
    NUMBER = {6},
     PAGES = {1231--1242},
      ISSN = {0010-3640,1097-0312},
   MRCLASS = {53C21 (53C20)},
  MRNUMBER = {3948556},
MRREVIEWER = {Thomas\ Schick},
       DOI = {10.1002/cpa.21803},
       URL = {https://doi.org/10.1002/cpa.21803},
}

@article {CCE,
    AUTHOR = {Carlotto, A. and Chodosh, O. and Eichmair, M.},
     TITLE = {Effective versions of the positive mass theorem},
   JOURNAL = {Invent. Math.},
  FJOURNAL = {Inventiones Mathematicae},
    VOLUME = {206},
      YEAR = {2016},
    NUMBER = {3},
     PAGES = {975--1016},
      ISSN = {0020-9910,1432-1297},
   MRCLASS = {53C20 (53C24 53C42)},
  MRNUMBER = {3573977},
MRREVIEWER = {Luc\ Nguyen},
       DOI = {10.1007/s00222-016-0667-3},
       URL = {https://doi.org/10.1007/s00222-016-0667-3},
}

@article {CheegerGromoll,
    AUTHOR = {Cheeger, J. and Gromoll, D.},
     TITLE = {The splitting theorem for manifolds of nonnegative {R}icci
              curvature},
   JOURNAL = {J. Differential Geometry},
  FJOURNAL = {Journal of Differential Geometry},
    VOLUME = {6},
      YEAR = {1971},
     PAGES = {119--128},
      ISSN = {0022-040X,1945-743X},
   MRCLASS = {53C20},
  MRNUMBER = {303460},
MRREVIEWER = {J.\ R.\ Vanstone},
       URL = {http://projecteuclid.org/euclid.jdg/1214430220},
}

@misc{AX24,
      title={New spectral Bishop-Gromov and Bonnet-Myers theorems and applications to isoperimetry}, 
      author={Antonelli, G. and Xu, K.},
      year={2024},
      eprint={2405.08918},
      archivePrefix={arXiv},
      primaryClass={math.DG},
      url={https://arxiv.org/abs/2405.08918}, 
}

@book {MaggiBook,
    AUTHOR = {Maggi, F.},
     TITLE = {Sets of finite perimeter and geometric variational problems},
    SERIES = {Cambridge Studies in Advanced Mathematics},
    VOLUME = {135},
      NOTE = {An introduction to geometric measure theory},
 PUBLISHER = {Cambridge University Press, Cambridge},
      YEAR = {2012},
     PAGES = {xx+454},
      ISBN = {978-1-107-02103-7},
   MRCLASS = {49-01 (26B20 28-02 49-02 49Q05 49Q20)},
  MRNUMBER = {2976521},
MRREVIEWER = {Giovanni\ Alberti},
       DOI = {10.1017/CBO9781139108133},
       URL = {https://doi.org/10.1017/CBO9781139108133},
}

@book {AmbrosioFuscoPallara,
    AUTHOR = {Ambrosio, L. and Fusco, N. and Pallara, D.},
     TITLE = {Functions of bounded variation and free discontinuity
              problems},
    SERIES = {Oxford Mathematical Monographs},
 PUBLISHER = {The Clarendon Press, Oxford University Press, New York},
      YEAR = {2000},
     PAGES = {xviii+434},
      ISBN = {0-19-850245-1},
   MRCLASS = {49-02 (49J45 49K10 49Qxx)},
  MRNUMBER = {1857292},
MRREVIEWER = {J.\ E.\ Brothers},
}

@article {Zhu23JDG,
    AUTHOR = {Zhu, J.},
     TITLE = {Rigidity results for complete manifolds with nonnegative
              scalar curvature},
   JOURNAL = {J. Differential Geom.},
  FJOURNAL = {Journal of Differential Geometry},
    VOLUME = {125},
      YEAR = {2023},
    NUMBER = {3},
     PAGES = {623--644},
      ISSN = {0022-040X,1945-743X},
   MRCLASS = {53C24 (53C21)},
  MRNUMBER = {4674077},
       DOI = {10.4310/jdg/1701804153},
       URL = {https://doi.org/10.4310/jdg/1701804153},
}

@article {Kasue83,
    AUTHOR = {Kasue, A.},
     TITLE = {Ricci curvature, geodesics and some geometric properties of
              {R}iemannian manifolds with boundary},
   JOURNAL = {J. Math. Soc. Japan},
  FJOURNAL = {Journal of the Mathematical Society of Japan},
    VOLUME = {35},
      YEAR = {1983},
    NUMBER = {1},
     PAGES = {117--131},
      ISSN = {0025-5645,1881-1167},
   MRCLASS = {53C20 (53C22)},
  MRNUMBER = {679079},
MRREVIEWER = {Yu.\ Burago},
       DOI = {10.2969/jmsj/03510117},
       URL = {https://doi.org/10.2969/jmsj/03510117},
}

@article {ChodoshliSoapBubbles24,
    AUTHOR = {Chodosh, O. and Li, C.},
     TITLE = {Generalized soap bubbles and the topology of manifolds with
              positive scalar curvature},
   JOURNAL = {Ann. of Math. (2)},
  FJOURNAL = {Annals of Mathematics. Second Series},
    VOLUME = {199},
      YEAR = {2024},
    NUMBER = {2},
     PAGES = {707--740},
      ISSN = {0003-486X,1939-8980},
   MRCLASS = {53C21 (53A10)},
  MRNUMBER = {4713021},
       DOI = {10.4007/annals.2024.199.2.3},
       URL = {https://doi.org/10.4007/annals.2024.199.2.3},
}

@article {KinnunenShanmugalingamQuasiminimalSets,
    AUTHOR = {Kinnunen, J. and Korte, R. and Lorent, A. and
              Shanmugalingam, N.},
     TITLE = {Regularity of sets with quasiminimal boundary surfaces in
              metric spaces},
   JOURNAL = {J. Geom. Anal.},
  FJOURNAL = {Journal of Geometric Analysis},
    VOLUME = {23},
      YEAR = {2013},
    NUMBER = {4},
     PAGES = {1607--1640},
      ISSN = {1050-6926,1559-002X},
   MRCLASS = {49Q20 (26B30 28A12)},
  MRNUMBER = {3107671},
MRREVIEWER = {Martin\ Fuchs},
       DOI = {10.1007/s12220-012-9299-z},
       URL = {https://doi.org/10.1007/s12220-012-9299-z},
}

@article {Fischer-Colbrie-Schoen,
    AUTHOR = {Fischer-Colbrie, D. and Schoen, R.},
     TITLE = {The structure of complete stable minimal surfaces in
              {$3$}-manifolds of nonnegative scalar curvature},
   JOURNAL = {Comm. Pure Appl. Math.},
  FJOURNAL = {Communications on Pure and Applied Mathematics},
    VOLUME = {33},
      YEAR = {1980},
    NUMBER = {2},
     PAGES = {199--211},
      ISSN = {0010-3640,1097-0312},
   MRCLASS = {53C40 (58E12)},
  MRNUMBER = {562550},
MRREVIEWER = {Themistocles\ M.\ Rassias},
       DOI = {10.1002/cpa.3160330206},
       URL = {https://doi.org/10.1002/cpa.3160330206},
}

@article {LiWangJDG,
    AUTHOR = {Li, P. and Wang, J.},
     TITLE = {Complete manifolds with positive spectrum},
   JOURNAL = {J. Differential Geom.},
  FJOURNAL = {Journal of Differential Geometry},
    VOLUME = {58},
      YEAR = {2001},
    NUMBER = {3},
     PAGES = {501--534},
      ISSN = {0022-040X,1945-743X},
   MRCLASS = {58J50 (53C20)},
  MRNUMBER = {1906784},
MRREVIEWER = {Man\ Chun\ Leung},
       URL = {http://projecteuclid.org/euclid.jdg/1090348357},
}

@article {LiWangCrelle,
    AUTHOR = {Li, P. and Wang, J.},
     TITLE = {Stable minimal hypersurfaces in a nonnegatively curved
              manifold},
   JOURNAL = {J. Reine Angew. Math.},
  FJOURNAL = {Journal f\"ur die Reine und Angewandte Mathematik. [Crelle's
              Journal]},
    VOLUME = {566},
      YEAR = {2004},
     PAGES = {215--230},
      ISSN = {0075-4102,1435-5345},
   MRCLASS = {53C42 (49Q05)},
  MRNUMBER = {2039328},
MRREVIEWER = {Denise\ M.\ Halverson},
       DOI = {10.1515/crll.2004.005},
       URL = {https://doi.org/10.1515/crll.2004.005},
}

@article {LiWangAsens,
    AUTHOR = {Li, P. and Wang, J.},
     TITLE = {Weighted {P}oincar\'e{} inequality and rigidity of complete
              manifolds},
   JOURNAL = {Ann. Sci. \'Ecole Norm. Sup. (4)},
  FJOURNAL = {Annales Scientifiques de l'\'Ecole Normale Sup\'erieure.
              Quatri\`eme S\'erie},
    VOLUME = {39},
      YEAR = {2006},
    NUMBER = {6},
     PAGES = {921--982},
      ISSN = {0012-9593},
   MRCLASS = {53C24 (35B40 35J60 35R45 53C21)},
  MRNUMBER = {2316978},
MRREVIEWER = {Harish\ Seshadri},
       DOI = {10.1016/j.ansens.2006.11.001},
       URL = {https://doi.org/10.1016/j.ansens.2006.11.001},
}

@book {HebeyBook,
    AUTHOR = {Hebey, E.},
     TITLE = {Nonlinear analysis on manifolds: {S}obolev spaces and
              inequalities},
    SERIES = {Courant Lecture Notes in Mathematics},
    VOLUME = {5},
 PUBLISHER = {New York University Courant Institute of Mathematical Sciences},
      YEAR = {1999},
     PAGES = {x+309},
   MRCLASS = {58D15 (35J60 46E35 53C21 58J60)},
  MRNUMBER = {1688256},
MRREVIEWER = {Gilles\ Carron},
}

@article {CaoShenZhu97,
    AUTHOR = {Cao, H.-D. and Shen, Y. and Zhu, S.},
     TITLE = {The structure of stable minimal hypersurfaces in {${\bf
              R}^{n+1}$}},
   JOURNAL = {Math. Res. Lett.},
  FJOURNAL = {Mathematical Research Letters},
    VOLUME = {4},
      YEAR = {1997},
    NUMBER = {5},
     PAGES = {637--644},
      ISSN = {1073-2780},
   MRCLASS = {53C21 (53C42 58E15)},
  MRNUMBER = {1484695},
MRREVIEWER = {William\ P.\ Minicozzi, II},
       DOI = {10.4310/MRL.1997.v4.n5.a2},
       URL = {https://doi.org/10.4310/MRL.1997.v4.n5.a2},
}

@article {ChodoshLiBernstein2,
    AUTHOR = {Chodosh, O. and Li, C.},
     TITLE = {Stable anisotropic minimal hypersurfaces in {${\bf R}^4$}},
   JOURNAL = {Forum Math. Pi},
  FJOURNAL = {Forum of Mathematics. Pi},
    VOLUME = {11},
      YEAR = {2023},
     PAGES = {Paper No. e3, 22},
      ISSN = {2050-5086},
   MRCLASS = {53C42 (35J50 53A10)},
  MRNUMBER = {4546104},
MRREVIEWER = {Piotr\ Rybka},
       DOI = {10.1017/fmp.2023.1},
       URL = {https://doi.org/10.1017/fmp.2023.1},
}

\typeout{get arXiv to do 4 passes: Label(s) may have changed. Rerun}
\end{document}